\documentclass[11pt,dvips]{article}
\usepackage{latexsym}
\usepackage{amsmath,amsthm,amsfonts,amssymb}
\usepackage[T1]{fontenc}
\usepackage[latin1]{inputenc}
\usepackage{aeguill}
\usepackage{url}
\usepackage{enumerate}
\usepackage{mdwlist} 
\usepackage{graphicx}
\usepackage[a4paper,textwidth=15cm,textheight=25cm]{geometry}
\usepackage{xspace}

\newtheorem{thm}{Theorem}[section]
\newtheorem{thmbis}{Theorem}
\newtheorem*{thm*}{Theorem}
\newtheorem{dfn}[thm]{Definition} 
\newtheorem*{dfn*}{Definition}

\newtheorem*{cor*}{Corollary}
\newtheorem{corbis}{Corollary}
\newtheorem{prop}[thm]{Proposition} 
\newtheorem*{prop*}{Proposition} 
\newtheorem*{properties*}{Properties} 
 
\newtheorem{lem}[thm]{Lemma} 
\newtheorem*{lem*}{Lemma}

\newtheorem*{claim*}{Claim} 
 
\newtheorem*{fact*}{Fact}

\newtheorem*{qst*}{Question}

\newtheorem*{pb*}{Problem}

\theoremstyle{remark}
 
\newtheorem*{algo*}{Algorithm} 
\newtheorem*{rem*}{Remark}

\newtheorem*{example*}{Example}

\newcounter{numEnonceTmpInterne}
\newenvironment{enonce*}[1]{\theoremstyle{plain}\stepcounter{numEnonceTmpInterne}%
\def\a{enoncetmp\alph{numEnonceTmpInterne}}%
\newtheorem*{\a}{#1}\begin{\a}}{\end{\a}}

\makeatletter
\edef\@tempa#1#2{\def#1{\mathaccent\string"\noexpand\accentclass@#2 }}
\@tempa\rond{017}
\makeatother

\renewcommand{\phi}{\varphi} 
\newcommand{\m} {^{-1}}

\newcommand {\ra} {\rightarrow}

\newcommand {\into} {\hookrightarrow}

\newcommand{\ie} {i.e.\ }

\newcommand {\calh} {{\mathcal {H}}}

\newcommand {\bbF} {{\mathbb {F}}}

\newcommand {\bbZ} {{\mathbb {Z}}}   

\newcommand{\grp}[1]{\langle #1 \rangle}

\newcommand{\Hom} {{\mathrm{Hom}}}

\newcommand{\gobble}[1]{} 

\newcommand {\F} {{\mathbb {F}}}  
 
\newcommand {\Z} {{\mathbb {Z}}}

\newcommand{\inc}{\subset}

\begin{document}

\title{Computing equations for residually free groups}
\author{Vincent Guirardel, Gilbert Levitt}
\date{
}

\maketitle

 \begin{abstract}
We show that there is no algorithm deciding whether the maximal residually free
quotient of a given finitely presented group is finitely presentable or
not.
 
 Given a finitely generated subgroup $G$  of a finite product of limit
 groups, we  discuss the possibility of finding an explicit set of
 defining equations (i.e. of expressing $G$ as the maximal residually free
 quotient of an explicit finitely presented group).
 \end{abstract}

\section{Introduction}

Any countable group $G$ has a largest residually free quotient $RF(G)$, equal to $G/\bigcap_{f\in \calh}\ker f$ where   $\calh$ 
is the set of   all homomorphisms from $G$ to a non-abelian free group $\F$.

In the language of   \cite{BMR_algebraicI}, if $R$ is a finite set of group equations on a finite set of variables $S$, then
$G=RF(\grp{S\mid R})$ is the \emph{coordinate group} of the variety defined by the system of equations $R$.  We say that $R$ is a \emph{set of defining equations} of $G$ over $S$.
Equational noetherianness of free groups implies that any finitely generated residually free group $G$
has a  (finite) set of  defining equations  \cite{BMR_algebraicI}.

On the other hand, any finitely generated residually free group embeds into a finite product
of limit groups (also known as finitely generated fully residually free groups),
which correspond to the \emph{irreducible components} of the variety defined by $R$
 \cite{BMR_algebraicI, KhMy_irreducible2, Sela_diophantine1}. Conversely, any subgroup of a finite product of limit groups is residually free.

This gives three possibilites to  define a finitely generated residually free group $G$ in an explicit way:
\begin{enumerate}
\item give a finite presentation of $G$ (if $G$ is finitely presented);
\item give a set of defining equations of $G$: write $G=RF(\grp{S\mid R})$, with $S$ and $R$ finite;
\item write $G$ as the subgroup of $L_1\times\dots\times L_n$ generated by a finite subset $S$,
where $L_1,\dots,L_n$ are limit groups given by some finite presentations.
\end{enumerate}

We investigate the algorithmic possibility to go back and forth between these ways of defining $G$.

One can go from 2 to 3:
given a set of defining equations of $G$, one can find 
  an explicit embedding into some product of limit groups 
\cite{KhMy_irreducible2,KhMy_effective, BHMS_fprf,GrWi_enumerating}.

Conversely, if $G$ is given as a subgroup of a product of limit groups, \emph{and if one knows that $G$ is finitely presented,} one can 
compute a presentation of $G$ \cite{BHMS_fprf}.  Obviously, a finite presentation is  a set of defining equations.

Since residually free groups are not always finitely presented, we investigate the following question:

\begin{qst*}
Let $L=L_1\times\dots\times L_n$ be a product of limit groups.
Let $G$ be the subgroup generated by a finite   subset    $S\inc L $.
Can one algorithmically find a finite set of defining equations for $G$, \ie find a finite presentation $\langle S \mid  R\rangle$ such that $G=RF(\langle S \mid  R\rangle)$?
\end{qst*}

We will prove that   this question has a negative answer.
On the other hand, we introduce  a closely related notion which  has  better algorithmic properties.

Let $RF_{na}(G)$ be the quotient $G/\bigcap_{f\in \calh_{na}}\ker f$ where   $\calh_{na}$ is the set of 
all homomorphisms from $G$ to $\F$ with    \emph{non-abelian image}.
Of course, $RF_{na}(G)$ is a quotient of $RF(G)$, which forgets the information about morphisms to $\bbZ$.  In fact (Lemma \ref{lem_centre}), it  is the quotient of $ RF(G)$ by its center.

We say that $G$ is a \emph{residually non-abelian free} group if
$G=RF_{na}(G)$, \ie if every non-trivial element of $G$ survives in a non-abelian free quotient of $G$;
equivalently, $G$ is   residually non-abelian free  
if and only if $G$ is residually free and has trivial center.
Given a residually non-abelian free group $G$, we say that $R$ 
is a \emph{set of $na$-equations} of $G$ over $S$ if $G=RF_{na}(\grp{S\mid R})$.

We write $Z(G)$ for the center of $G$, and $b_1(G)$ for the torsion-free rank of $H_1(G,\Z)$. 

\begin{thmbis}\label{thm1}
  \begin{itemize}
  \item There is an algorithm which takes as input presentations of limit groups $L_1,\dots,L_n$, and a finite subset  $S 
    \subset  L_1\times\dots\times L_n$, and which computes a finite set of $na$-equations for $G/Z(G)=RF_{na}(G)$, where $G=\langle S\rangle$.

  \item 
    One can compute a finite set of defining equations for  $G=\langle S\rangle$ if and only if one can compute $b_1(G)$.

  \end{itemize}
\end{thmbis}

Since there is no algorithm computing $b_1(\langle S\rangle)$ from $S\inc \F_2\times \F_2$ \cite{BrMi_structure}, we deduce:

\begin{corbis}\label{cor2}
There is no algorithm which takes as an input a finite subset  $S\inc \F_2\times \F_2$ and computes a finite set of equations for $\langle S\rangle$.  \qed
\end{corbis}

We also investigate the possiblity to decide whether a residually free quotient is finitely presented. Using  Theorem \ref{thm1} and \cite{Grunewald_fp}, we prove:

\begin{thmbis}\label{thm2}
There is no algorithm with takes as an input a finite group presentation $\grp{S\mid R}$, 
and which decides whether $RF(\grp{S\mid R})$ is finitely presented.
\end{thmbis}

\section{The residually non-abelian free quotient $RF_{na}$}

We always  denote by $G$  a finitely generated group, and by $\F$     a non-abelian free group.

\begin{dfn}

$RF(G)$ is  the quotient of $G$ by 
the intersection of the kernels of all morphisms   $G\to \F$. 

$RF_{na}(G)$ is the quotient of $G$ by 
the intersection of the kernels of all morphisms   $G\to \F$ with non-abelian image.
\end{dfn}

One may view $RF(G)$ as the image of $G$ in $\bbF^\calh$, where $\calh$ is the set of   all morphisms  $G\to \F$, and $RF_{na}(G)$ as the image in $\bbF^{\calh_{na}}$, where $\calh_{na}$ is the set of all   morphisms with non-abelian image.

 Every  homomorphism $G\to \F$ factors through $RF(G)$ (through $RF_{na}(G)$ if its image is not abelian). By definition,    $G$ is residually free if and only if $G=RF(G)$,  residually non-abelian free if and only if $G=RF_{na}(G)$.

\begin{lem}\label{lem_centre}
There is an exact sequence
$$1\ra Z(RF(G))\ra RF(G)\ra RF_{na}(G)\ra 1 .$$
\end{lem}

In particular, $G$ is   residually non-abelian free   if and only if $G$ is residually free and $Z(G)=1$.
If $G$ is a non-abelian limit group, it has trivial center and  $RF_{na}(G)=RF(G)=G$.

\begin{proof}
Let $H=RF(G)$.
Consider $a\in Z(H)$ and $f:H\ra\bbF$ with $f(a)\neq 1$.
The image of $f$ centralizes $f(a)$, so is abelian by commutative transitivity of $\bbF$. 
Thus $a$ has trivial image in $RF_{na}(H)=RF_{na}(G)$.

Conversely, consider $a\in H\setminus Z(H)$, and $b\in H$ with $[a,b]\neq 1$.
There exists $f:H\ra \bbF$ such that $f([a,b])\neq 1$.
Then $f(H)$ is non-abelian, and $f(a)\neq 1$. This means that the image of $a$ in $RF_{na}(G)$ is non-trivial.
\end{proof}

Any epimorphism $f:G\ra H$ induces epimorphisms $f_{RF}:RF(G)\ra RF(H)$
 and
  $f_{na}:RF_{na}(G)\ra RF_{na}(H)$. 

\begin{lem}\label{lem_b1}
Let $f:G\ra H$ be an epimorphism.
Then $f_{RF}:RF(G)\ra RF(H)$ is an isomorphism if and only if $f_{na}:RF_{na}(G)\ra RF_{na}(H)$ is an isomorphism
and $b_1(G)=b_1(H)$.
\end{lem}

\begin{proof} Note that $f_{RF}$ (resp.\ $f_{na}$) is an isomorphism if and only if any morphism $G\ra \bbF$ (resp.\ any such morphism with non-abelian image)
factors through $f$.
The lemma then follows from the fact  that the embedding $\Hom(H, \Z)\into \Hom(G, \Z)$ induced by $f$  is onto   if and only if $b_1(G)=b_1(H)$.
\end{proof}

Given a product $L_1\times\dots\times L_n$, we denote by $p_i$ the projection onto $L_i$. 

\begin{lem}\label{lem_nonab}
  Let $G\subset L=L_1\times\dots\times L_n$ with $L_i$ a limit group.
  Let $I\subset\{1,\dots,n\}$ be the set of indices such that $p_i(G)$ is abelian.
Then $RF_{na}(G)$ is the image of $G$ in $L'=\prod_{i\notin I} L_i$ (viewed as a quotient of $L_1\times \dots\times L_n$).
\end{lem}

\begin{proof} Note that $G=RF(G)$. An element $(x_1,\dots,x_n)\in G$ is in $Z(G)$ if and only if $x_i$ is central in $p_i(G)$ for every $i$. Since $p_i(G)$ is abelian or has trivial center, $Z(G)$ is the kernel of the natural projection $L\to L'$. The result follows from Lemma \ref{lem_centre}. 
\end{proof}

\begin{lem}\label{lem_fp}
  $RF(G)$ is finitely presented if and only if $RF_{na}(G)$ is.
\end{lem}

\begin{proof}
If $H$ is any residually free group, the abelianization map $H\to H_{ab}$ is injective on $Z(H)$ since any element of $Z(H)$ survives in some free quotient of $H$, which has to be cyclic (see \cite[Lemma 6.2]{ BHMS_fprf}). In particular, $Z(H)$ is finitely generated if $H$ is. Applying this to $H=R(G)$, the exact sequence of Lemma \ref{lem_centre} gives the required result. 
\end{proof}

\section{Proof of the theorems}

Let $S$ be a finite set of elements in a group. We define $S_0=S\cup\{1\}$. If $R,R'$ are sets of words on $S\cup S\m$,
then $R^{S_0}$ is the set of all words obtained by conjugating elements of $R$ by   elements of $S_0$,
and $[R^{S_0},R']$ is the set of all words obtained as commutators of  words in $R^{S_0}$ and   words in $R'$.

\begin{prop}\label{prop_eqna} Let $A_1,\dots,A_n$ be arbitrary groups,  with $n\ge2$.
  Let $G\subset A_1\times\dots \times A_n$ be  generated by $S=\{s_1,\dots,s_k\}$.
Let 
 $p_i:G\ra A_i$  be the projection. 
Assume that $p_i(G)=RF_{na}(\grp{S\mid R_i})$ for some finite set of relators $R_i$.

Then the     set  
 $$ \Tilde R= [R_n^{S_0},[R_{n-1}^{S_0},\dots [R_3^{S_0}, [R_2^{S_0},R_1]]\dots]]$$
is a finite set of $na$-equations of $RF_{na}(G)$ over $S$, \ie $RF_{na}(G)=RF_{na}(\grp{S\mid \Tilde R })$.

\end{prop}

An equality such as $p_i(G)=RF_{na}(\grp{S\mid R_i})$ means that there is an isomorphism commuting with the natural projections $F(S)\to p_i(G)$ and $F(S)\to RF_{na}(\grp{S\mid R_i})$, where $F(S)$ denotes the free group on $S$.

\begin{proof}
Recall that a free group $\bbF$ is CSA: 
commutation is transitive on $\bbF\setminus\{1\}$,
and maximal abelian subgroups are malnormal. 
In particular, if two non-trivial subgroups commute, then both are abelian.
If $A, B$   are non-trivial   subgroups of $\bbF$, and if $A$ commutes with $B,B^{x_1},\dots,B^{x_p}$ for elements $x_1,\dots ,x_p\in \bbF$, 
then $\grp{A,B,x_1,\dots ,x_p}$ is   abelian.

We write 
$$\Tilde G=\grp{S\mid \Tilde R}=\grp{S\mid  [ R_n^{S_0},[R_{n-1}^{S_0},\dots[R_2^{S_0},R_1]\dots]]}.$$
 We always denote by $\varphi:F(S)\to\F$ a morphism with non-abelian image. We shall show that 
\emph{such a $\varphi$   factors through $G$ if and only if it factors through $\tilde G$.} 
This implies the desired result $RF_{na}(G)=RF_{na}( \Tilde G )$: both groups are equal to the image of $F(S)$ in $\bbF^{\calh_{na}}$, 
where $\calh_{na}$ is the set of all $\varphi$'s  which factor through $G$ and $\tilde G$.

We proceed by induction on $n$. 
We first claim that $\varphi$ is trivial on $\Tilde R$ if and only if it is trivial on some $R_i$. The if direction is clear. 
For the only if direction, observe that the image of $[R_{n-1}^{S_0},\dots [R_2^{S_0},R_1]\dots]$ 
commutes with all conjugates of $\varphi(R_n)$ by elements of $\varphi(F(S))$, 
so $R_n$ or $[R_{n-1}^{S_0},\dots[R_2^{S_0}, R_1]\dots]$ has trivial image. The claim follows by induction.

Now suppose that $\varphi$ factors through $\Tilde G$. Then $\varphi$ kills $\Tilde R$, hence some $R_i$. It follows that $\varphi$ factors through $p_i(G)$, hence through $G$. 

Conversely, suppose that $\varphi$ factors through $f:  G\to \F$. 
Consider the intersection of $G$ with the kernel of $p_n:G\to A_n$ and the kernel of $p_{1,\dots,n-1}:G\to A_1\times \dots\times A_{n-1}$. These are commuting normal subgroups of $G$. If both have non-trivial image in $\F$, the CSA property implies that the image of $f$ is abelian, a contradiction. We deduce that $f$ factors through $p_n $ or through $p_{1,\dots,n-1}$, and by induction that it factors through some $p_i$. Thus  $\varphi$ kills $R_i$, hence $\Tilde R$ as required. 
\end{proof}

\begin {proof}
[Proof of Theorem \ref{thm1}]

Let $L_1, \dots, L_n$ and $G=\langle S\rangle$ be as in Theorem \ref{thm1}.   
Using a solution of the word problem in a limit group, 
one can find the indices $i$ for which $p_i(G)\inc L_i$ is abelian (this amounts to checking whether the elements of $p_i(S)$ commute).

First assume that no   $p_i(G)$ is  abelian. 
As pointed out in \cite[Lemma 7.5]{BHMS_fprf}, one deduces from \cite{Wilton_hall}  an algorithm yielding  a finite presentation 
$  \langle S \mid  R_i\rangle$ of $p_i(G)$. Since $p_i(G)$ is not abelian, one has  $p_i(G)=RF_{na}(\langle S \mid  R_i\rangle)$,
 and Proposition \ref{prop_eqna} yields a finite set of na-equations for $RF_{na}(G)$ over $S$   (if $n=1$, then $RF_{na}(G)=p_1(G)$).  
If some $p_i(G)$'s are  abelian, we simply replace $G$ by its image in   $L'$ as in Lemma \ref{lem_nonab}. This proves the first assertion of the theorem.

Now suppose that $b_1(G)$ is known. 
We want a finite set $R$ such that $RF(G)=RF(\grp{S\mid  R})$.
 If $n=1$,   then $G$ is a  subgroup of the limit group $L_1$, and one can find a finite presentation of $G$ as explained above.
So assume $n\geq 2$.
  Consider the finite presentation $\Tilde G=\grp{S\mid \Tilde R}$ given by  Proposition \ref{prop_eqna},
so that $RF_{na}(\Tilde G)=RF_{na}(G)$. 

We claim that $G$ is a quotient of $\Tilde  G$. To see this,   we consider an $x\in F(S)$ which is trivial in $\Tilde G$ and we prove that it is trivial in $G$. If not,  residual freeness of $G$ implies that $x$ survives under a morphism $\varphi:F(S)\to \F$ which factors through $G$. 
If $\varphi$ has non-abelian image, it factors through $RF_{na}(G)=RF_{na}(\Tilde G)$, hence through $\Tilde G$, 
contradicting the triviality of $x$ in $\Tilde G$. If the image is abelian, $\varphi$ also factors through $\Tilde G$ because 
  all relators in $\Tilde R$ are commutators.

Since $\Tilde R$ is finite, we can compute $b_1(\Tilde G)$.
If $b_1(\Tilde G)=b_1(G)$, we are done by Lemma \ref{lem_b1}  since $G$ is a quotient of $\Tilde  G$.
If $b_1(\Tilde G)>b_1(G)$, we enumerate all trivial words of $G$ (using an enumeration of trivial words in each $p_i(G)$),
and we add them to the presentation of $\Tilde G$ one by one. 
We compute    $b_1$ after each addition, and we stop when we reach the  known value $b_1(G)$.

Conversely, if  we have a finite set of defining equations for $G$, so that $G=RF(\langle S \mid  R \rangle)$, we can compute  $b_1(\langle S \mid  R \rangle)$, which equals $b_1(G)$ by Lemma \ref{lem_b1}.
\end{proof}

\begin{thm}[Theorem 2]
There is no algorithm with takes as input a finite group presentation $\grp{S\mid \Tilde R}$, 
and which decides whether $ RF(\grp{S\mid \Tilde R})$ is finitely presented.
\end{thm}

\begin{proof} 
Given a finite set $S\subset \bbF_2\times \bbF_2$, Theorem \ref{thm1} provides a finite set $\Tilde R$ such that $RF_{na}(\langle S\rangle)=RF_{na}(\grp{S\mid \Tilde R })$. Using Lemma \ref{lem_fp}, we see that finite presentability of $RF (\grp{S\mid \Tilde R })$ is equivalent to that of $RF_{na}(\grp{S\mid \Tilde R })$, hence to that of $RF (\langle S\rangle)=\langle S\rangle$. But  it follows from \cite{Grunewald_fp} that  there is no algorithm which decides, given a finite set  $S\subset \bbF_2\times \bbF_2$, whether $\langle S\rangle$ is finitely presented. 
\end{proof}

\bibliographystyle{alpha2}
\bibliography{published,unpublished}

\begin{flushleft}
Vincent Guirardel\\
Institut de Math\'ematiques de Toulouse\\
Universit\'e de Toulouse et CNRS (UMR 5219)\\
118 route de Narbonne\\
F-31062 Toulouse cedex 9\\
France.\\
\emph{e-mail:}\texttt{guirardel@math.ups-tlse.fr}\\[8mm]

Gilbert Levitt\\
Laboratoire de Math\'ematiques Nicolas Oresme\\
Universit\'e de Caen et CNRS (UMR 6139)\\
BP 5186\\
F-14032 Caen Cedex\\
France\\
\emph{e-mail:}\texttt{levitt@math.unicaen.fr}\\
\end{flushleft}

\end{document}